\documentclass{amsart}
\usepackage{mathptmx, amssymb, enumerate, colonequals, amscd, url, xcolor, color}
\definecolor{chianti}{rgb}{0.6,0,0}
\definecolor{meretale}{rgb}{0,0,.6}
\definecolor{leaf}{rgb}{0,.35,0}
\usepackage[colorlinks=true, pagebackref, hyperindex, citecolor=meretale, urlcolor=leaf, linkcolor=chianti]{hyperref}

\newtheorem{theorem}{Theorem}[section]
\newtheorem{lemma}[theorem]{Lemma}
\newtheorem{proposition}[theorem]{Proposition}

\theoremstyle{definition}
\newtheorem{example}[theorem]{Example}
\newtheorem{remark}[theorem]{Remark}
\newtheorem{question}[theorem]{Question}
\numberwithin{equation}{theorem}

\def\ge{\geqslant}
\def\le{\leqslant}
\def\lf{\lfloor}
\def\rf{\rfloor}
\def\segre{\,\#\,}
\def\to{\longrightarrow}
\def\mapsto{\longmapsto}
\def\into{\lhook\joinrel\longrightarrow}

\def\divv{\operatorname{div}}

\def\Proj{\operatorname{Proj}}

\def\fraka{\mathfrak{a}}

\def\frakm{\mathfrak{m}}

\def\AA{\mathbb{A}}

\def\FF{\mathbb{F}}
\def\NN{\mathbb{N}}
\def\PP{\mathbb{P}}
\def\QQ{\mathbb{Q}}
\def\ZZ{\mathbb{Z}}

\def\calO{\mathcal{O}}
\def\calS{\mathcal{S}}

\begin{document}
\title{On Segre products, $F$-regularity, and finite Frobenius representation type}

\author{Anurag K. Singh}
\address{Department of Mathematics, University of Utah, 155 South 1400 East, Salt Lake City, UT~84112, USA}
\email{singh@math.utah.edu}

\author{Kei-ichi Watanabe}
\address{Department of Mathematics, College of Humanities and Sciences, Nihon University, Setagaya-ku, Tokyo, 156-8550, Japan,
and\newline\indent
Organization for the Strategic Coordination of Research and Intellectual Properties, Meiji University, Higashimita 1-1-1, Tama-ku, Kawasaki, 214-8571, Japan}
\email{watnbkei@gmail.com}

\thanks{A.K.S.~was supported by NSF grants DMS~1801285 and DMS~2101671, and K.W. by JSPS Grant-in-Aid for Scientific Research~20K03522. This paper started from conversations at the \emph{Advanced Instructional School on commutative algebra and algebraic geometry in positive characteristics}, Indian Institute of Technology, Bombay. The authors take this opportunity to thank the program organizers for their hospitality.}

\dedicatory{Dedicated to Ngo Viet Trung on the occasion of his 70th birthday,\\
in celebration of his many contributions to commutative algebra}

\begin{abstract}
We study the behavior of various properties of commutative Noetherian rings under Segre products, with a special focus on properties in positive prime characteristic defined using the Frobenius endomorphism. Specifically, we construct normal graded rings of finite Frobenius representation type that are not Cohen-Macaulay.
\end{abstract}
\maketitle

\section{Introduction}

We study the behavior of various properties of commutative Noetherian rings under Segre products, with a special focus on properties in positive prime characteristic defined using the Frobenius endomorphism. Segre products of rings arise rather naturally in the context of projective varieties: while the product of affine spaces $\AA^{\!m}$ and $\AA{\!^n}$ is readily identified with $\AA^{\!m+n}$, it is the \emph{Segre embedding} that gives the product of projective spaces~$\PP^m$ and $\PP^n$ the structure of a projective variety: 
\[
\PP^m\times\PP^n\to\PP^{m+n+mn},\qquad \big((a_0,\dots,a_m), (b_0,\dots,b_n)\big)\mapsto(a_0b_0,a_0b_1,\dots,a_mb_n).
\]
At the level of homogeneous coordinate rings, this corresponds to
\[
\PP^m\times\PP^n\ =\ \Proj\FF[x_0y_0,\ x_0y_1,\ \dots,\ x_my_n],
\]
where $\PP^m\colonequals\Proj\FF[x_0,\dots,x_m]$ and $\PP^n\colonequals\Proj\FF[y_0,\dots,y_n]$.

More generally, for $\NN$-graded rings $R=\oplus_{n\ge 0}R_n$ and $S=\oplus_{n\ge 0}S_n$, finitely generated over a field $R_0=\FF=S_0$, the \emph{Segre product} of $R$ and $S$ is the $\NN$-graded ring
\[
R\segre S\colonequals\bigoplus_{n\ge 0}R_n\otimes_\FF S_n.
\]
It is readily seen that $R\segre S$ is a subring of the tensor product $R\otimes_\FF S$; moreover, $R\segre S$ is a direct summand of $R\otimes_\FF S$ as an $R\segre S$-module, equivalently the inclusion of rings
\[
R\segre S \into R\otimes_\FF S
\]
is pure; it follows from this that if $\FF$ is a field of positive characteristic, and $R$ and $S$ are~$F$-pure or $F$-regular, then the same is also true for $R\segre S$. What is perhaps surprising is that the converse also holds, provided that the $\NN$-grading on each of the rings $R$ and~$S$ is irredundant; this is proved here as Theorem~\ref{theorem:f:regular}, see also~\cite[Theorem~5.2]{Hashimoto:multigraded}. The additional hypothesis on the grading is indeed required in view of Example~\ref{example:redundant}.

While the properties $F$-purity and $F$-regularity are inherited by pure subrings, the property of being $F$-rational is not, as established by the second author in~\cite{Watanabe:boutot}. Nonetheless, we show that if $R$ and $S$ are $F$-rational rings of positive prime characteristic, then $R\segre S$ is also $F$-rational, Theorem~\ref{theorem:f:rational}. The converse to this is false, see Example~\ref{example:f:rational}.

Lastly, we turn to the property of finite Frobenius representation type (FFRT); the notion is due to Smith and Van den Bergh~\cite{Smith-VdB}, and it follows readily from their results that if~$R$ and $S$ are $\NN$-graded reduced rings, finitely generated over a perfect field $R_0=\FF=S_0$ of positive characteristic, then $R\segre S$ has FFRT. We use this to construct normal graded rings that are not Cohen-Macaulay, but have the FFRT property.

The observation that Segre products readily yield large families of normal graded rings that are not Cohen-Macaulay goes back at least to Chow~\cite{Chow}, who established necessary and sufficient conditions for the Segre product of Cohen-Macaulay rings to be Cohen-Macaulay; Hochster and Roberts~\cite[\S 14]{HR:invariants} observed that under mild hypotheses, Chow's results may be recovered via the K\"unneth formula for sheaf cohomology. Subsequently, Goto and Watanabe~\cite{GW} established a more general K\"unneth formula for local cohomology that extends this circle of ideas; this and other ingredients are summarized next.

\section{Preliminaries}

We first record the K\"unneth formula for local cohomology, \cite[Theorem~4.1.5]{GW}:

\begin{theorem}[Goto-Watanabe]
Let $R$ and $S$ be normal $\NN$-graded rings, finitely generated over a field $R_0=\FF=S_0$. Set $\frakm_R$, $\frakm_S$, and $\frakm$ to be the homogeneous maximal ideals of the rings $R$, $S$, and~$R\segre S$ respectively. Suppose~$M$ and $N$ are finitely generated $\ZZ$-graded modules over $R$ and $S$ respectively, such that $H^k_{\frakm_R}(M)=0=H^k_{\frakm_S}(N)$ for $k=0,1$.

Then, for each $k\ge 0$, the local cohomology of the $\ZZ$-graded $R\segre S$-module
\[
M \segre N\colonequals\bigoplus_{n\in\ZZ}M_n\otimes_\FF N_n
\]
is given by
\[
H^k_{\frakm}(M\segre N)\ =\ \big(M\segre H^k_{\frakm_S}(N)\big) \oplus \big(H^k_{\frakm_R}(M)\segre N\big) \oplus \bigoplus_{i+j=k+1} \big(H^i_{\frakm_R}(M)\segre H^j_{\frakm_S}(N)\big).
\]
\end{theorem}

Our proof of Theorem~\ref{theorem:f:regular} uses the description of normal graded rings in terms of $\QQ$-divisors, due to Dolga\v cev~\cite{Dolgachev}, Pinkham~\cite{Pinkham}, and Demazure~\cite{Demazure}, that we review next. A \emph{$\QQ$-divisor} on a normal projective variety $X$ is a $\QQ$-linear combination of codimension one irreducible subvarieties of $X$. Let $D=\sum n_iV_i$ be a $\QQ$-divisor, where $n_i\in\QQ$, and the subvarieties $V_i$ are distinct. Set
\[
\lf D\rf\colonequals\sum\lf n_i\rf V_i,
\]
where $\lf n\rf$ is the greatest integer less than or equal to $n$. We define
\[
\calO_X(D)\colonequals\calO_X(\lf D\rf).
\]
Let $K(X)$ denote the field of rational functions on $X$. Each $g\in K(X)$ defines a Weil divisor~$\divv(g)$ by considering the zeros and poles of $g$ with appropriate multiplicity. As these multiplicities are integers, it follows that for a $\QQ$-divisor $D$ one has
\begin{multline*}
H^0(X,\calO_X(\lf D\rf)) = \{g\in K(X)\mid\divv(g)+\lf D\rf \ge 0 \} \\
=\ \{g\in K(X)\mid\divv(g)+ D\ge 0 \} = H^0(X,\calO_X(D)).
\end{multline*}

A $\QQ$-divisor $D$ is \emph{ample} if $ND$ is an ample Cartier divisor for some $N\in\NN$. In this case, the \emph{generalized section ring} $\varGamma_*(X,D)$ is the $\NN$-graded ring 
\[
\varGamma_*(X,D)\colonequals\bigoplus_{n\ge0}H^0(X,\calO_X(nD))T^n,
\]
where $T$ is an element of degree $1$, transcendental over $K(X)$. The following is \cite[3.5]{Demazure}:

\begin{theorem}[Demazure]
\label{theorem:Demazure}
Let $R$ be an $\NN$-graded normal domain that is finitely generated over a field $R_0$. Let $T$ be a homogeneous element of degree $1$ in the fraction field of $R$. Then there exists a unique ample $\QQ$-divisor $D$ on $X\colonequals\Proj R$ such that
\[
R_n\ =\ H^0(X,\calO_X(nD))T^n\quad\text{ for each }n\ge 0.
\]
\end{theorem}

Let $D=\sum(s_i/t_i)V_i$ be a $\QQ$-divisor where the $V_i$ are distinct, $s_i$ and $t_i$ are relatively prime integers, and $t_i>0$. Following \cite[Theorem~2.8]{Watanabe:demazure}, the \emph{fractional part} of $D$ is
\[
D'\colonequals\sum\frac{t_i-1}{t_i}V_i.
\]
This definition is motivated by the fact that one then has
\[
-\lf -nD\rf\ =\ \lf D'+nD\rf
\]
for each integer $n$, so that taking the graded dual of
\[
{[H^{\dim R}_\frakm(R)]}_{-n}\ =\ H^{\dim X}(X,\calO_X(-nD))
\]
yields
\[
{[\omega_R]}_n\ =\ H^0(X,\calO_X(K_X+D'+nD)),
\]
where $\omega_R$ is the graded canonical module of $R\colonequals\varGamma_*(X,D)$, and $K_X$ is the canonical divisor of $X$. The following is \cite[Theorem~3.3]{Watanabe:dim2}; note that
\[
H^{\dim X}(X,\calO_X(K_X+D'))\ =\ H^{\dim X}(X,\calO_X(K_X))
\]
is the rank one vector space ${[H^{\dim R}_\frakm(\omega_R)]}_0$.

\begin{theorem}[Watanabe]
\label{theorem:Watanabe}
Let $X$ be a normal projective variety of characteristic $p>0$, and $K_X$ its canonical divisor. Let $D$ be an ample $\QQ$-divisor, and set $R\colonequals\varGamma_*(X,D)$. Then:
\begin{enumerate}[\ \rm(i)]
\item The ring $R$ is $F$-pure if and only if the Frobenius map below is injective:
\[
F\colon H^{\dim X}(X,\calO_X(K_X+D'))\to H^{\dim X}(X,\calO_X(pK_X+pD')).
\]
\item Let $\eta$ be a nonzero element of $H^{\dim X}(X,\calO_X(K_X+D'))$. Then the ring $R$ is $F$-regular if and only if for each integer $n>0$ and each nonzero element $c$ of $H^0(X,\calO_X(nD))$, there exists an integer $e>0$ such that $cF^e(\eta)$ is a nonzero element of 
\[
H^{\dim X}(X,\calO_X(p^e(K_X+D')+nD)).
\]
\end{enumerate}
\end{theorem}

\section{$F$-regularity and $F$-purity}

The theory of tight closure was introduced by Hochster and Huneke in \cite{HH:JAMS}, and further developed in the graded context in~\cite{HH:JAG}. A ring $R$ of positive prime characteristic is \emph{weakly $F$-regular} if each ideal of $R$ equals its tight closure, while $R$ is \emph{$F$-regular} if each localization of $R$ is weakly $F$-regular. Following \cite[page 166]{Hochster:tc}, a ring $R$ of positive prime characteristic is \emph{strongly $F$-regular} if $N_M^*=N$ for each pair of $R$-modules $N\subseteq M$. When~$R$ is an $\NN$-graded ring that is finitely generated over a field $R_0$ of positive characteristic, as is the case in the present paper, the properties of weak $F$-regularity, $F$-regularity, and strong~$F$-regularity coincide by \cite[Corollary~3.3]{Lyubeznik-Smith}.

The following theorem may be viewed as an extension of~\cite[Theorem~5.2]{Hashimoto:multigraded}, where it was proved under the hypothesis that the  rings contain homogeneous elements of degree~$1$:

\begin{theorem}
\label{theorem:f:regular}
Let $R$ and $S$ be normal $\NN$-graded rings, finitely generated over a perfect field $R_0=\FF=S_0$ of positive characteristic. Suppose that the fraction fields of $R$ as well as~$S$ contain homogeneous elements of degree $1$.

Then the Segre product $R\segre S$ is $F$-regular (respectively, $F$-pure) if and only if $R$ and $S$ are F-regular (respectively, $F$-pure).
\end{theorem}

\begin{proof}
If the rings $R$ and $S$ are $F$-regular or $F$-pure, then the same holds for their tensor product~$R\otimes_\FF S$, see for example the proof of $2\implies 3$ in~\cite[Theorem~5.2]{Hashimoto:multigraded}. The property, then, is inherited by the pure subring $R\segre S$; it is only the converse that requires the additional hypothesis on the grading:

Let $D_X$ and $D_Y$ be ample $\QQ$-divisors on $X\colonequals\Proj R$ and $Y\colonequals\Proj S$ respectively, such that~$R=\varGamma_*(X,D_X)$ and $S=\varGamma_*(Y,D_Y)$. Set $Z\colonequals X\times Y$, and let $\pi_1\colon Z\to X$ and $\pi_2\colon Z\to Y$ be the respective projection morphisms. For each integer $n\ge 0$ one has
\[
H^0(Z,\calO_Z(\pi_1^*(nD_X)+\pi_2^*(nD_Y)))\ =\ H^0(X,\calO_X(nD_X))\otimes H^0(Y,\calO_X(nD_Y)),
\]
from which it follows that
\[
R\segre S\ =\ \varGamma_*(Z,\pi_1^*(D_X)+\pi_2^*(D_Y)).
\]
Setting $D_Z\colonequals\pi_1^*(D_X)+\pi_2^*(D_Y)$, one has
\[
D_Z'\ =\ \pi_1^*(D_X') + \pi_2^*(D_Y'),
\]
so the Frobenius map~$F$ as in Theorem~\ref{theorem:Watanabe}~(i) takes the form
\[
\CD
H^{d_1+d_2}(Z,\calO_Z(K_Z+D_Z')) @>\cong>> H^{d_1}(X,\calO_X(K_X+D_X'))\otimes H^{d_2}(Y,\calO_Y(K_Y+D_Y'))\\
@VFVV @VFVV\\
H^{d_1+d_2}(Z,\calO_Z(pK_Z+pD_Z')) @>\cong>> H^{d_1}(X,\calO_X(pK_X+pD_X'))\otimes H^{d_2}(Y,\calO_Y(pK_Y+pD_Y'))
\endCD
\]
where $d_1\colonequals\dim X$ and $d_2\colonequals\dim Y$. Let $\eta_1$ and $\eta_2$ be nonzero elements of the rank one vector spaces $H^{d_1}(X,\calO_X(K_X+D'))$ and $H^{d_2}(Y,\calO_Y(K_Y+D_Y'))$ respectively.

If $R\segre S$ is $F$-pure, the injectivity of the vertical arrows in the diagram displayed above implies that $F(\eta_1\otimes\eta_2)=F(\eta_1)\otimes F(\eta_2)$ is nonzero, and hence that the maps
\[
\CD
H^{d_1}(X,\calO_X(K_X+D_X')) @>F>> H^{d_1}(X,\calO_X(pK_X+pD_X'))
\endCD
\]
and
\[
\CD
H^{d_2}(Y,\calO_Y(K_Y+D_Y')) @>F>> H^{d_2}(Y,\calO_Y(pK_Y+pD_Y'))
\endCD
\]
are injective; it follows that the rings $R$ and $S$ are $F$-pure.

Next, assume that $R\segre S$ is $F$-regular. Fix $n>0$, and consider nonzero elements
\[
c_1\in H^0(X,\calO_X(nD_X))\qquad\text{ and }\qquad c_2\in H^0(Y,\calO_Y(nD_Y)).
\]
Then $c_1\otimes c_2$ is a nonzero element of $H^0(Z,\calO_X(nD_Z))$, so the $F$-regularity of $R\segre S$ implies that there exists $e>0$ such that
\[
(c_1\otimes c_2)F^e(\eta_1\otimes\eta_2)\ =\ c_1F^e(\eta_1)\otimes c_2F^e(\eta_2)
\]
is a nonzero element of
\begin{multline*}
H^{d_1+d_2}(Z,\calO_Z(p^e(K_Z+D_Z')+nD_Z))\\
\cong\
H^{d_1}(X,\calO_X(p^e(K_X+D_X')+nD_X))\otimes H^{d_2}(Y,\calO_Y(p^e(K_Y+D_Y')+nD_Y)).
\end{multline*}
But then the elements
\[
c_1F^e(\eta_1)\ \in\ H^{d_1}(X,\calO_X(p^e(K_X+D_X')+nD_X))
\]
and
\[
c_2F^e(\eta_2)\ \in\ H^{d_2}(Y,\calO_Y(p^e(K_Y+D_Y')+nD_Y))
\]
are nonzero, implying that the rings $R$ and $S$ are $F$-regular.
\end{proof}

The hypothesis that the $\NN$-grading on $R$ and $S$ is irredundant is indeed required:

\begin{example}
\label{example:redundant}
Consider the hypersurface $R\colonequals\FF_{\!2}[x,y,z]/(x^2+y^3+z^3)$ where $x,y,z$ have degrees $3,2,2$, respectively, and $S\colonequals\FF_{\!2}[u,v]$ where $u$ and $v$ have degree $2$. The ring $R$ is not~$F$-pure or $F$-regular since the element $x$ belongs to the Frobenius closure of the ideal~$(y,z)R$. However, since the ring $S$ is supported only in even degrees, one has
\[
R\segre S\ =\ R^{(2)}\segre S\ =\ \FF_{\!2}[y,z]\segre\FF_{\!2}[u,v]\ =\ \FF_{\!2}[uy,uz,vy,vz],
\]
which is $F$-regular. Note that while the fraction field of $R$ contains homogeneous elements of degree $1$, the fraction field of $S$ does not.
\end{example}

\section{$F$-rationality}

Following \cite[page~125]{Hochster:tc}, a local ring of positive prime characteristic is \emph{$F$-rational} if it is a homomorphic image of a Cohen-Macaulay ring, and each ideal generated by a system of parameters is tightly closed; a Noetherian ring of positive prime characteristic is \emph{$F$-rational} if its localization at each maximal ideal (equivalently, at each prime ideal) is~$F$-rational. With this definition, an $F$-rational ring is normal and Cohen-Macaulay.

For the case of interest in this paper, let $R$ be an $\NN$-graded normal domain that is a finitely generated algebra over a field $R_0$ of positive characteristic. Then $R$ is $F$-rational if and only if the ideal generated by some (equivalently, any) homogeneous system of parameters for~$R$ is tightly closed; see~\cite[Theorem~4.7]{HH:JAG} and the preceding remark.

Smith~\cite{Smith:ratsing} proved that $F$-rational rings have rational singularities; the converse, more precisely the theorem that rings with rational singularities have $F$-rational type, is due independently to Hara~\cite{Hara:AJM} and to Mehta and Srinivas~\cite {Mehta-Srinivas}. 

Let $R$ be a finitely generated algebra over a field of characteristic zero; Boutot's theorem states that if $R$ has rational singularities, then so does each pure subring of $R$, \cite{Boutot}. The corresponding statement for $F$-rational rings turns out to be false: in~\cite{Watanabe:boutot} the second author constructed an example of an $F$-rational ring with a pure subring that is not $F$-rational. Nonetheless, we have:

\begin{theorem}
\label{theorem:f:rational}
Suppose $R$ and $S$ are $F$-rational $\NN$-graded rings, finitely generated over a perfect field $R_0=\FF=S_0$ of positive characteristic. Then $R\segre S$ is $F$-rational. 
\end{theorem}

\begin{proof}
Note that $R$ and $S$ are Cohen-Macaulay; it suffices to assume that they have positive dimension, in which case $a(R)<0$ and $a(S)<0$ by \cite[Satz~3.1]{Flenner} or \cite[Theorem~2.2]{Watanabe:k:star}. Using this, the K\"unneth formula shows that $R\segre S$ is Cohen-Macaulay and that
\begin{equation}
\label{equation:lc}
H^d_\frakm(R\segre S)\ =\ H^{\dim R}_{\frakm_R}(R)\segre H^{\dim S}_{\frakm_S}(S),
\end{equation}
where $d\colonequals\dim(R\segre S)$, and $\frakm_R$, $\frakm_S$, and $\frakm$ are the homogeneous maximal ideals of the rings~$R$, $S$, and~$R\segre S$ respectively. The hypothesis that $\FF$ is perfect ensures that the ring~$R\segre S$ is normal. By~\cite[Corollary~6.8]{Hashimoto:cmfi}, the ring $R\otimes_\FF S$ is $F$-rational.
 
It suffices to show that the zero submodule of~\eqref{equation:lc} is tightly closed. Suppose, to the contrary, that $c$ and $\eta$ are nonzero homogenous elements of $R\segre S$ and~$H^d_\frakm(R\segre S)$ respectively, with $cF^e(\eta)=0$ in $H^d_\frakm(R\segre S)$ for $e\gg0$. It follows that $cF^e(\eta)$ is also zero for~$e\gg0$, when regarded as an element of
\[
H^{\dim R}_{\frakm_R}(R)\otimes_\FF H^{\dim S}_{\frakm_S}(S).
\]
But then $\eta$, regarded as an element of the module above, is in the tight closure of zero; this contradicts the $F$-rationality of $R\otimes_\FF S$.
\end{proof}

In contrast with Theorem~\ref{theorem:f:regular}, $R\segre S$ may be $F$-rational even when $R$ and $S$ are not:

\begin{example}
\label{example:f:rational}
Let $\FF$ be a field of positive characteristic, and consider the hypersurfaces
\[
R\colonequals\FF[x,y,z]/(x^2+y^3+z^7)\qquad\text{ and }\qquad S\colonequals\FF[u,v,w]/(u^4+v^5+w^5),
\]
with $x,y,z$ having degrees $21,14,6$ respectively, and $u,v,w$ having degrees $5,4,4$ respectively. Then $a(R)=1$ and $a(S)=7$, so $R$ and $S$ are not $F$-rational. Note that the gradings are irredundant, i.e., as in the hypotheses of Theorem~\ref{theorem:f:regular}, the fraction fields of $R$ as well as $S$ contain homogeneous elements of degree $1$.

Since ${[H^2_{\frakm_R}(R)]}_{\ge 0}$ is supported only in degree $1$, and ${[H^2_{\frakm_S}(S)]}_{\ge 0}$ in degrees $2$, $3$, and~$7$, the K\"unneth formula shows that $R\segre S$ is Cohen-Macaulay, and also that $a(R\segre S)=-5$. Suppose that the characteristic of $\FF$ is at least $7$. Then the Frobenius action on each of
\[
{[H^2_{\frakm_R}(R)]}_{\le -5}\qquad\text{ and }\qquad {[H^2_{\frakm_S}(S)]}_{\le -5}
\]
and hence on $H^2_{\frakm_R}(R)\segre H^2_{\frakm_S}(S)$ is injective. Moreover, we claim that $R\segre S$ has an isolated non $F$-regular point: to see this, let $r\otimes s$ be a nonzero homogeneous element of $R\segre S$ of positive degree; then the ring
\[
(R\segre S)_{r\otimes s}\ =\ R_r\segre S_s
\]
is a pure subring of the regular ring $R_r\otimes_\FF S_s$, and is hence $F$-regular. It follows that $R\segre S$ is $F$-rational by \cite[Theorem~7.1]{HH:JAG}.
\end{example}

\section{Finite Frobenius representation type}

The notion of rings of finite Frobenius representation type (FFRT) is due to Smith and Van den Bergh; it is an essential ingredient in their proof of the following remarkable theorem: If $R$ is a graded direct summand of a polynomial ring over a perfect field $\FF$ of positive characteristic, then the ring of $\FF$-linear differential operators on $R$ is a simple ring, see~\cite[Theorem~1.3]{Smith-VdB}. This is striking in that the corresponding statement is not known for polynomial rings over fields of characteristic zero.

Subsequently, the FFRT property has found several other applications: Seibert \cite{Seibert} proved that over rings with FFRT, the Hilbert-Kunz multiplicity is rational; tight closure commutes with localization for rings with FFRT by Yao \cite{Yao}; if $R$ is a Gorenstein ring with FFRT, Takagi and Takahashi \cite{Takagi-Takahashi} proved that each local cohomology module of the form~$H^k_\fraka(R)$ has finitely many associated primes; the Gorenstein hypothesis may be removed, as proved subsequently by Hochster and N\'u\~nez-Betancourt \cite{Hochster-NB}.

A reduced ring $R$ of positive prime characteristic $p$, satisfying the Krull-Schmidt theorem, is said to have \emph{finite Frobenius representation type} if there exists a finite set~$\calS$ of~$R$-modules such that for each~$q=p^e$, each indecomposable summand of $R^{1/q}$ is isomorphic to an element of $\calS$. When $R$ is Cohen-Macaulay, each indecomposable summand of~$R^{1/q}$ is a maximal Cohen-Macaulay $R$-module; thus, Cohen-Macaulay rings of finite representation type have FFRT, though the latter property is much weaker: e.g., in the graded setting, the FFRT property is inherited by direct summands, \cite[Proposition~3.1.6]{Smith-VdB}.

Key examples of rings with FFRT include those that are graded direct summands of polynomial rings; such rings are also $F$-regular, and hence Cohen-Macaulay. Recent work on the FFRT property includes that of Hara and Ohkawa~\cite{HO}, where they study the property for $2$-dimensional normal graded rings in terms of $\QQ$-divisors, and~\cite{RSV1, RSV2} where Raedschelders, \v Spenko, and Van den Bergh prove that over an algebraically closed field of characteristic $p\ge\max\{n-2,3\}$, the Pl\"ucker homogeneous coordinate ring of the Grassmannian~$G(2,n)$ has FFRT.

Our goal here is to construct normal rings with FFRT that are not Cohen-Macaulay. Note that a Stanley-Reisner ring over a perfect field has FFRT by \cite[Example~2.36]{Kamoi}, though such a ring need not be Cohen-Macaulay. Our interest here, however, is primarily in normal domains. We first record:

\begin{lemma}
Let $\FF$ be a perfect field of positive characteristic, and let $R$ and $S$ be reduced rings that are finitely generated $\FF$-algebras. Suppose, moreover, that $R$, $S$, and $R\otimes_\FF S$ satisfy the Krull-Schmidt theorem. Then, if $R$ and $S$ have FFRT, so does $R\otimes_\FF S$.
\end{lemma}

\begin{proof}
If $R$ and $S$ have FFRT, there exist indecomposable $R$-modules $M_1,\dots,M_m$, and indecomposable $S$-modules~$N_1,\dots,N_n$ such that for each $q=p^e$, one has
\[
R^{1/q}\ \cong\ \bigoplus M_i\quad\text{ and }\quad S^{1/q}\ \cong\ \bigoplus N_j
\]
where, in each case, the index set depends on $q$, and modules may be repeated within the direct sum. Set $T\colonequals R\otimes_\FF S$. Then 
\[
T^{1/q}\ \cong\ R^{1/q}\otimes_\FF S^{1/q}\ \cong\ \big(\bigoplus M_i\big)\otimes_\FF\big(\bigoplus N_j\big)\ \cong\ \bigoplus\big(M_i \otimes_\FF N_j\big).
\]
Each of the $mn$ modules of the form $M_i \otimes_\FF N_j$ is a direct sum of finitely many indecomposable $T$-modules. This provides a finite set of indecomposable $T$-modules that contains an isomorphic copy of each indecomposable summand of $T^{1/q}$ for each $q=p^e$.
\end{proof}

\begin{proposition}
\label{proposition:ffrt}
Let $R$ and $S$ be $\NN$-graded reduced rings, finitely generated over a perfect field $R_0=\FF=S_0$ of positive characteristic. If $R$ and $S$ have FFRT, then the rings~$R\otimes_\FF S$ and~$R\segre S$ also have FFRT.
\end{proposition}

\begin{proof}
The statement regarding the tensor product follows immediately from the lemma, bearing in mind that the Krull-Schmidt theorem holds for $\NN$-graded rings $A$ with $A_0$ a field.

The assertion about the Segre product follows from \cite[Proposition~3.1.6]{Smith-VdB}, since $R\segre S$ is a graded direct summand of the tensor product $R\otimes_\FF S$.
\end{proof}

\begin{example}
\label{example:ffrt}
Let $\FF$ be a perfect field of characteristic $p\ge 7$, and consider the hypersurface~$R\colonequals\FF[x,y,z]/(x^2+y^3-z^p)$, with $x,y,z$ having degrees $3p,2p,6$ respectively. Note that the ring $R$ is sandwiched between $A\colonequals\FF[x,y]$ and $A^{1/p}=\FF[x^{1/p},y^{1/p}]$, since
\[
z\ =\ x^{2/p}+y^{3/p}.
\]
As $A$ is a polynomial ring, and hence has finite representation type, it follows that $R$ has~FFRT by~\cite[Observation~3.7, Theorem~3.10]{Shibuta}. Set $S\colonequals\FF[u,v]$, where $u$ and $v$ are indeterminates with degree $1$. Then the ring $R\segre S$ has FFRT by Proposition~\ref{proposition:ffrt}. However, since~$a(R)=p-6>0$, the K\"unneth formula shows that $R\segre S$ is not Cohen-Macaulay.
\end{example}

\begin{remark}
The examples above are characteristic-specific: to illustrate, let~$p\ge 7$ be a prime integer, and let $\FF$ now be an \emph{arbitrary} field. Set $\PP^1\colonequals\Proj\FF[u,v]$, with points of~$\PP^1$ parametrized by $u/v$. If $p=6k+1$, consider the $\QQ$-divisor
\begin{equation}
\label{equation:D:1}
D\colonequals\frac{1}{2}(0)-\frac{1}{3}(\infty)-\frac{k}{p}(-1).
\end{equation}
Then $\varGamma_*(\PP^1,D)\colonequals\bigoplus H^0(\PP^1,nD)T^n$ is the $\FF$-algebra generated by 
\[
z\colonequals\frac{v^2(u+v)}{u^3}T^6, \quad y\colonequals\frac{v^{4k+1}(u+v)^{2k}}{u^{6k+1}}T^{2p}, \quad x\colonequals\frac{v^{6k+1}(u+v)^{3k}}{u^{9k+1}}T^{3p},
\]
where $T$ is an indeterminate of degree one. It is readily seen that $\varGamma_*(\PP^1,D)$ is a hypersurface with defining equation $z^p=x^2+y^3$.

If $p=6k-1$, consider instead the $\QQ$-divisor
\begin{equation}
\label{equation:D:2}
D\colonequals\frac{1}{3}(\infty)+\frac{k}{p}(-1)-\frac{1}{2}(0).
\end{equation}
In this case, $\varGamma_*(\PP^1,D)$ is the $\FF$-algebra generated by 
\[
z\colonequals\frac{u^3}{v^2(u+v)}T^6, \quad y\colonequals\frac{u^{6k-1}}{v^{4k-1}(u+v)^{2k}}T^{2p}, \quad x\colonequals\frac{u^{9k-1}}{v^{6k-1}(u+v)^{3k}}T^{3p}.
\]
Once again, $\varGamma_*(\PP^1,D)$ is a hypersurface with defining equation $z^p=x^2+y^3$.

Note that the denominators occurring in the $\QQ$-divisor $D$ in~\eqref{equation:D:1} and~\eqref{equation:D:2} are $2$, $3$, and $p$. It follows from \cite[Theorem~7.2]{HO} that if the characteristic of $\FF$ is not $2$, $3$, or $p$, then the hypersurface~$\FF[x,y,z]/(x^2+y^3-z^p)$ does not have FFRT.
\end{remark}

This raises the question:

\begin{question}
\label{question:ffrt:cm}
Let $R$ be a normal graded domain, finitely generated over a field of characteristic zero. If $R$ has dense FFRT type, i.e., there exists a dense set of prime integers~$p$ for which the mod $p$ reductions $R_p$ have FFRT, then is $R$ a Cohen-Macaulay ring?
\end{question}

A related question is the following; see also \cite[Question~9.1]{Mallory:ffrt}.

\begin{question}
\label{question:ffrt:sfr}
Let $R$ be a normal graded domain, finitely generated over a field of characteristic zero. If $R$ has dense FFRT type, then is $R$ an $F$-regular ring?
\end{question}

The converse is false: \cite[Theorem~5.1]{Singh-Swanson} provides an example of an $F$-regular hypersurface $R$, over a field of characteristic zero, for which each mod $p$ reduction $R_p$ has a local cohomology module of the form $H^3_I(R_p)$ that has infinitely many associated prime ideals; it follows from \cite[Theorem~3.9]{Takagi-Takahashi} or \cite[Theorem~5.7]{Hochster-NB} that, for each prime integer $p$, the mod $p$ reduction $R_p$ does not have FFRT.

\section*{Acknowledgments}

The authors thank Mitsuyasu Hashimoto and Karl Schwede for useful discussions.


\end{document}